\documentclass[11pt]{amsart}

\usepackage{amsfonts,epsfig}
\usepackage{latexsym}
\usepackage{amssymb}
\usepackage{amsmath}
\usepackage{amsthm}
\usepackage{graphics}
\usepackage{verbatim}
\usepackage{enumerate}
\usepackage[all]{xy}
\usepackage{multirow}
\usepackage{array,booktabs,ctable}
\usepackage[backref]{hyperref}
\usepackage{color}
\definecolor{mylinkcolor}{rgb}{0.8,0,0}
\definecolor{myurlcolor}{rgb}{0,0,0.8}
\definecolor{mycitecolor}{rgb}{0,0,0.8}
\hypersetup{colorlinks=true,urlcolor=myurlcolor,citecolor=mycitecolor,linkcolor=mylinkcolor,linktoc=page,breaklinks=true}
\usepackage{bigstrut}

\newtheorem{defn}{Definition}[section]
\newtheorem{definition}[defn]{Definition}

\newtheorem{lemma}[defn]{Lemma}

\newtheorem{thm}[defn]{Theorem}
\newtheorem{theorem}[defn]{Theorem}

\newtheorem{prop}[defn]{Proposition}

\theoremstyle{definition}

\newtheorem{remark}[defn]{Remark}
\newtheorem{question}[defn]{Question}

\newcommand{\C}{\mathbb C}

\newcommand{\QQ}{\mathbb Q}
\newcommand{\Q}{\mathbb Q}

\newcommand{\ZZ}{\mathbb Z}
\newcommand{\Z}{\mathbb Z}

\newcommand{\PP}{\mathbb P}

\newcommand{\SL}{\operatorname{SL}}
\newcommand{\PSL}{\operatorname{PSL}}

\newcommand{\Gal}{\operatorname{Gal}}
\newcommand{\Aut}{\operatorname{Aut}}

\newcommand{\GL}{\operatorname{GL}}

\renewcommand{\Im}{\operatorname{Im}}

\newcommand{\tor}{\mathrm{tors}}

\makeatletter
\@namedef{subjclassname@2020}{\textup{2020} Mathematics Subject Classification}
\makeatother
\begin{document}



\bibliographystyle{plain}
\title{Minimal Subgroups of $\GL_2(\ZZ_{S})$}

\author{Harris B. Daniels}
\address{Department of Mathematics and Statistics, Amherst College, Amherst, MA 01002, USA}
\email{hdaniels@amherst.edu}
\urladdr{http://www3.amherst.edu/~hdaniels/}

\author{Jeremy Rouse}
\address{Department of Mathematics, Wake Forest University, Winston-Salem, NC 27109, USA}
\email{rouseja@wfu.edu}
\urladdr{https://users.wfu.edu/rouseja/}
{}

\keywords{Elliptic Curves,  Galois Representations, Profinite Groups.}

\subjclass[2020]{Primary: 11G05, Secondary: 11F80,  14H52, 22E50.}

\begin{abstract}
Let $E$ be an elliptic curve over a number field $L$ and for a finite set $S$ of primes, let
$\rho_{E,S} : \Gal(\overline{L}/L) \to \GL_{2}(\ZZ_{S})$ be the $S$-adic Galois representation. If $L \cap \Q(\zeta_{n}) = \Q$ for all positive integers $n$ whose prime factors are in $S$, then $\det \rho_{E,S} : \Gal(\overline{L}/L) \to \ZZ_{S}^{\times}$ is surjective. We say that a finite index subgroup $H \subseteq \GL_{2}(\ZZ_{S})$ is \emph{minimal} if $\det : H \to \ZZ_{S}^{\times}$ is surjective, but $\det : K \to \ZZ_{S}^{\times}$ is not surjective for any proper closed subgroup $K$ of $H$. We show that there are no minimal subgroups of $\GL_{2}(\ZZ_{S})$ unless $S = \{ 2 \}$, while minimal subgroups of $\GL_{2}(\ZZ_{2})$ are plentiful. We give models for all the genus $0$ modular curves associated to minimal subgroups of $\GL_{2}(\ZZ_{2})$, and construct an infinite family of elliptic curves over imaginary quadratic fields with bad reduction only at $2$ and with minimal $2$-adic image.
\end{abstract}
\maketitle

\section{Introduction}\label{sec:intro}

Given an elliptic curve $E/\QQ$, a prime $p$, and a fixed algebraic closure of $\overline\QQ$, one can construct the $p$-adic Galois representation associated to $E$, \[
\rho_{E,p^\infty}:\Gal(\overline\QQ/\QQ)\to\GL_2(\ZZ_p).
\]
Recently there has been great interest in studying the groups of the form $G_{E,p} = \Im(\rho_{E,p^\infty})$. 
A few classical results about $G_{E,p}$ include that $\det(G_{E,p}) = \ZZ_p^\times$ and if $E$ does not have complex multiplication, then $[\GL_2(\ZZ_p):G_{E,p}]$ is finite (see Chapter IV of \cite{Serre68}). 

Given these two things, a natural question to ask is if there is an elliptic curve $E/\QQ$ and prime $p$ such that $G_{E,p}$ has the property that if $H\subsetneq G_{E,p}$, then $\det(H)\subsetneq \ZZ_p^\times$? Before answering this question, we phrase it more broadly.

Let $S$ be a finite set of primes. We also let \[\ZZ_S = \lim_{\leftarrow} \ZZ/n\ZZ\] with respect to divisibility, but restricting $n$ to be divisible only by primes in $S$. 
In this case we have that
\[
\GL_2(\ZZ_S) = \prod_{p\in S}\GL_2(\ZZ_p). 
\]

\begin{defn}\label{def:min}
	A group $H\subseteq \GL_2(\ZZ_S)$ of finite index is called {\bfseries minimal} if $\det(H) = \ZZ_S^\times$, but for every maximal closed subgroup $M\subsetneq H$ we have that $\det(M)\subsetneq \ZZ_S^\times.$
\end{defn}

So a more general question would be if there is an elliptic curve $E/\QQ$ and a set of primes $S \subseteq \ZZ$ such that the image of the $S$-adic Galois representation $\Im\rho_{E,S} = G_{E,S} \subseteq \prod_{p\in S}G_{E,p}\subseteq \GL_2(\ZZ_S)$ is minimal and if so, can we classify all of, or maybe almost all of, such curves?

In order to start examining this question, we start by looking for examples. 

\subsection{First Examples}\label{sub:1Ex}
We claim that if $E/\QQ$ is an elliptic curve that does not have complex multiplication and $E$ only has bad reduction at 2, then $G_{E,2}$ is a minimal group. 

Suppose that $E/\QQ$ is an elliptic curve without complex multiplication that only has bad reduction at 2. 
From the criterion of N\'eron-Ogg-Shafarevich \cite{SerreTateNOS} we know that the extension $\QQ(E[2^\infty])/\QQ$ is only ramified at 2. 
From the main theorem of \cite{Tate94} we know that there are no such cubic extensions of $\QQ$. 
This, together with the fact that $\GL_2(\ZZ/2\ZZ)\simeq S_3$ would force $E$ to have a point of order 2 defined over $\QQ$. 
Besides being able to conclude that $E$ has to have a rational point of order 2, this allows us to conclude that $G_{E,2}$ is in fact a pro 2-group. 
This is because once we know that $E$ has a rational point of order 2, we know that $[\QQ(E[2]):\QQ]=1$ or $2$ and $[\QQ(E[2^{n+1}]) : \QQ(E[2^n])]$ is always a power of 2.
Now that we know that $G_{E,2}$ is a pro 2-group, we can use a classical result in group theory that says that any proper maximal subgroup of a pro $p$-group is normal and has index $p$. (Here and elsewhere, we require a maximal subgroup of a profinite group to be closed.)

Thus, in our case, the maximal proper subgroups of $G_{E,2}$ all have index 2 and are normal. But the Galois correspondence says that these maximal subgroups correspond to quadratic subfields of $\QQ(E[2^\infty]).$
Again, from the main theorem of \cite{Tate94} the only possible quadratic subfields of $\QQ(E[2^\infty])$ are subfields of $\QQ(\zeta_8)=\QQ(i,\sqrt{2})$. 
From this we know that every maximal subgroup of $G_{E,2}$ fixes one of these fields. Lastly, using the Weil pairing we see that none of the maximal subgroups of $G_{E,2}$ can have surjective determinant. 

Searching \cite{lmfdb} we find that there are 8 elliptic curves defined over $\QQ$ without complex multiplication that only have bad reduction at 2. 
These elliptic curves have labels 
\begin{center}
\href{https://lmfdb.org/EllipticCurve/Q/128/a/1}{\texttt{128.a1}}, 
\href{https://lmfdb.org/EllipticCurve/Q/128/a/2}{\texttt{128.a2}}, 
\href{https://lmfdb.org/EllipticCurve/Q/128/a/3}{\texttt{128.a3}}, 
\href{https://lmfdb.org/EllipticCurve/Q/128/a/4}{\texttt{128.a4}}, 
\href{https://lmfdb.org/EllipticCurve/Q/128/b/1}{\texttt{128.b1}}, 
\href{https://lmfdb.org/EllipticCurve/Q/128/b/2}{\texttt{128.b2}}, 
\href{https://lmfdb.org/EllipticCurve/Q/128/b/3}{\texttt{128.b3}} and 
\href{https://lmfdb.org/EllipticCurve/Q/128/b/4}{\texttt{128.b4}}.
\end{center}
\begin{remark}
We know that these are in fact all of the elliptic curves over $\QQ$ with bad reduction only at 2. This is because any elliptic curve with only bad reduction at two has conductor equal to a power of 2, but there is a bound on the exponent that can appear on 2 from \cite[Theorem 10.2]{Silv2}. Thus any elliptic curve over $\QQ$ with bad reduction only at 2 has conductor bounded by 256. So the completeness of the data in \cite{lmfdb} allows us to conclude that these are all of them. 
\end{remark}

Examining these curves further in the LMFDB, we see that they all have different $2$-adic images and the modular curves associated to each of these 8 different subgroups of $\GL_2(\ZZ_2)$ are genus 0. 
Using the labels conventions established in \cite{RSZB}, these groups are 
\begin{center}
\texttt{32.96.0.1}, 
\texttt{32.96.0.3}, 
\texttt{32.96.0.25}, 
\texttt{32.96.0.27},
\texttt{32.96.0.102}, 
\texttt{32.96.0.104}, 
\texttt{32.96.0.106}, and 
\texttt{32.96.0.108}. 
\end{center}
The models associated to each of the corresponding modular curves have been computed in \cite{RZB}\footnote{In this paper the curves are labeled \texttt{X238a-d} and \texttt{X239a-d}}.
So, our pure thought argument that had initially led to 8 examples has in fact led us to 8 separate infinite families of elliptic curves whose $2$-adic images are minimal. Thus it seems that this phenomenon might be quite common.

With these examples in hand, we move on to trying to better understand this phenomenon abstractly. 

\subsection{Statement of results and outline}

Our first main result is that minimal groups are a $2$-adic phenomena.
\begin{thm}\label{thm:S=2}
Let $S$ be a finite nonempty set of primes and $H\subseteq \GL_2(\ZZ_S)$ a minimal group of finite level. Then $S = \{2\}$.
\end{thm} 
This, together with {}the complete classification of possible 2-adic images available in \cite{RZB} completes the classification of elliptic curves defined over $\QQ$ with minimal image. 

On the other hand, minimal subgroups of $\GL_{2}(\Z_{2})$ are plentiful.
\begin{thm}\label{thm:lots}
Let $H \leq \GL_{2}(\Z_{2})$ be any finite-index subgroup with $\det(H) = \Z_{2}^{\times}$. Then there is
a minimal subgroup $M \leq H$ with $|H : M| < \infty$. 
\end{thm}
Our approach for proving the above result is the following. If $A$ and $B$ are randomly chosen elements of $H$, let $\langle A, B \rangle$ be the smallest closed subgroup containing $A$ and $B$. We show that
if $\det(\langle A, B \rangle) = \Z_{2}^{\times}$, then $\langle A, B \rangle$ is a minimal subgroup of $H$ with probability $1$.

Finally, while there are only $8$ non-CM elliptic curves $E/\Q$ with bad reduction only at $2$, we can give
an infinite family of elliptic curves $E$ defined over quadratic extensions of $\Q$ with bad reduction only at $2$ that have minimal $2$-adic image.

\begin{prop}
\label{prop:quadfamily}
Suppose that $n$ is a positive integer and let $a = \sqrt{-(2^{n}+1)}$. Let
\[
E : y^{2} = x^{3} + 2ax^{2} + (a^{2} + 1)x.
\]
Then $E$ has bad reduction only at prime ideals above $2$ in $\Q(a)$. 
If $n$ is odd and $n \ne 3$, then $\rho_{E,2^{\infty}}(G_{\Q(a)})$ is minimal and has RSZB label \texttt{8.24.0.86}. If $n = 2$, then $\rho_{E,2^{\infty}}(G_{\Q(a)})$ is minimal and has RSZB label \texttt{16.384.9.895}. If $n = 10$, then $\rho_{E,2^{\infty}}(G_{\Q(a)})$ is minimal and has RSZB label \texttt{16.384.9.894}.
\end{prop}

An outline of the paper is as follows. In Section~\ref{sec:2adic} we prove Theorem~\ref{thm:S=2} and discuss
a generalization to principally polarized abelian varieties. In Section~\ref{sec:plentiful} we prove Theorem~\ref{thm:lots} using ideas from $p$-adic Lie theory. In Section~\ref{sec:Gen0} we describe a search to find all the genus 0 minimal groups and then compute models for the corresponding modular curves. The models for these curves were computed using \cite{Magma}, the techniques in \cite{rakvi}, and guided by the information available in \cite{lmfdb}. All the code for these computations is available at \cite{github}. Lastly, we prove Proposition~\ref{prop:quadfamily} in Section~\ref{sec:ExQF}. 

\subsection{Acknowledgements}

We would first like to thank the anonymous referee for their thoughtful comments on a previous version of this paper. The second author is thankful to Amherst College for hospitality during a visit in September 2022.

\section{Background}\label{sec:Back}
The goal of this section is to establish notation and remind the reader of the basic facts necessary for the results in this paper. 
For more detail about elliptic curves, readers should see \cite{Silv1,Silv2}. 
For more information about modular curves the reader should see \cite{ALRBook1}.
Lastly, for more information about $p$-adic Lie theory see \cite{BookProP}.

\subsection{Elliptic Curves}

Elliptic curves are defined as smooth projective genus 1 curves with a specified point. 
They are ubiquitous in mathematics and can be found in the center of many of the open problems in modern number theory.

One of the most interesting aspects of elliptic curves is that given an elliptic curve $E$ defined over a number field $K$, the set of $K$-rational points on $E$ can be given the structure of a finitely generated abelian group. That is to say, that $E(K)\simeq \ZZ^r\oplus T,$
for some $r\in\ZZ_{\geq0}$ and finite abelian group $T$.

We call $r$ the {\it rank} of $E$ over $K$, and $T$ the torsion subgroup of $E(K)$, often denoted $E(K)_\tor.$ A classical result in the study of elliptic curves is that if we fix an algebraic closure of $K$, denoted $\overline{K}$, then
$$E[n] = \{P \in E(\overline{K}) : nP = \mathcal{O}\}\simeq (\ZZ/n\ZZ)^2.$$
Further, there is a natural componentwise action of $\Gal(\overline{K}/K)$ on $E[n]$ that induces a representation
$$\bar\rho_{E,n}\colon\Gal(\overline{K}/K)\to \GL_2(\ZZ/n\ZZ).$$

We call $\bar\rho_{E,n}$ the {\it mod $n$ representations associated with $E$}. Using the mod $n$ representations associated to $E$ and the appropriate inverse limits, we can define the $p$-adic and adelic Galois representations attached to $E$. We denote these representations
\[
\rho_{E,p^\infty}\colon\Gal(\overline{K}/K)\to \GL_2(\ZZ_p), \hbox{ and}
\]
\[
\rho_{E}\colon \Gal(\overline{K}/K) \to \GL_2(\widehat\ZZ)\simeq \prod_p \GL_2(\ZZ_p).
\]

Here we note these representations all depend on which points we pick as a basis for the $n$-torsion of $E(K)$,
and changing the basis replaces the image with a conjugate subgroup. It is not hard to see that if $H \subseteq \GL_{2}(\ZZ_{S})$ is minimal, then every conjugate of $H$ in $\GL_{2}(\ZZ_{S})$ is also minimal.

Given an elliptic curve $E/K$, we denote by $j(E)$ the usual $j$-invariant of $E$. If $E^{D}$ is the quadratic twist of $E$ by $D$ (i.e. if $E : y^{2} = x^{3} + ax^{2} + bx + c$, then $E^{D} : y^{2} = x^{3} + Dax^{2} + D^{2}bx^{3} + D^{3}c$), then $j(E) = j(E^{D})$. Conversely, if $E_{1}$ and $E_{2}$ are two elliptic curves defined over $K$ with $j(E_{1}) = j(E_{2}) \not\in \{ 0, 1728 \}$, then $E_{2}$ is a quadratic twist of $E_{1}$.

\subsection{Modular Curves}

Modular curves are an important tool for studying elliptic curves and their Galois representations. The points on these curves correspond to elliptic curves whose mod $n$ Galois representations are contained inside of a particular subgroup of $\GL_2(\ZZ/n\ZZ)$ up to conjugation. The goal of this subsection is to give the basic definitions and theorems needed for this paper. The interested reader is encouraged to see \cite{ALRBook1}. For the remainder of the section $n$ will be an integer greater than 1.



Associated to each group $G\subseteq \GL_2(\ZZ/n\ZZ)$ such that $\det(G) = (\ZZ/n\ZZ)^\times$, and $-I\in G$ is a modular curve $X_G$. 
The curve $X_G$ is a smooth, projective, and geometrically integral curve defined over $\QQ$. We say that $G$ has genus $g$ if $X_{G}$ is a curve with genus $g$.

Again, assuming that the group $G$ contains $-I$, the curve $X_G$ comes with a natural map 
\[
\pi_G: X_G\to \PP^1_\QQ
\]
called the $j$-map associate to $G$, such that if $E/K$ is an elliptic curve with $j(E) \not\in \{ 0, 1728 \}$, then $\Im\bar\rho_{E,n}$ is conjugate to a subgroup of $G$ if and only if there is a $P\in X_G(K)$ such that $j(E) = \pi_G(P)$.

In the case that $-I\not\in G$, $X_{G}$ is an algebraic stack and its course space is the same as $X_{\tilde{G}}$, where $\tilde{G} = \left\langle G,-I \right\rangle$. In this case the moduli interpretation is different; if $-I \not \in G$, if the image of $\bar\rho_{E,n} \subseteq G$, this does not imply the same for the quadratic twist $E_{D}$. In this situation, if $U$ is the complement in $X_{\tilde{G}}$ of the cusps and preimages of $j = 0$ and $1728$, there is a universal elliptic curve $\mathcal{E} \to U$ so that $E$ has image contained in $G$ if and only if there is some (not necessarily unique) $t$ so that $E \cong \mathcal{E}_{t}$ over $K$. For more detail see \cite[Section 2 and 5]{RZB}. 

\begin{remark}
A careful reading of the moduli interpretation, one notices that an elliptic curve $E/\QQ$ having a corresponding point on $X_G(\QQ)$ does not ensure that $\Im\bar\rho_{E,n}$ is conjugate to $G$. It only ensures that $\Im\bar\rho_{E,n}$ is conjugate to a subgroup of $G$. That said, if $G$ is a minimal group containing all matrices $\equiv I \pmod{n}$, since $\det \circ \bar\rho_{E,n}$ is surjective, it must be that $\Im\bar\rho_{E,n} = G.$
\end{remark}

There has been an immense amount of progress on computing modular curves $X_G$, their $j$-maps $\pi_G$, and their rational points. The work here relies directly or indirectly on the previous work and so we mention some of that work here. In particular, \cite{RZB,RSZB,Zyw} all make major contributions to our understanding the modular curves associates to subgroups of $\GL_2(\ZZ/\ell^k\ZZ)$ for a prime $\ell$ and integer $k\geq 1.$ Work has now started on understanding how these images can occur simultaneously and how they fit together \cite{DM,DLRM,DGJ,Zyw2}.

\section{Minimality is a 2-adic phenomena}\label{sec:2adic}

Notice that if $H$ is a minimal subgroup of $\GL_2(\ZZ_S)$, then every maximal open subgroup of $H$ can be obtained as the inverse image of a subgroup of $\ZZ_S^\times$ under the map $\det:H\to\ZZ_S^\times$. 
This is because if $M_1$ and $M_2$ were maximal subgroups of $H$ such that $\det(M_1) = \det(M_2)\subsetneq\ZZ_S^\times,$ then $\det^{-1}(\det(M_1)) = \det^{-1}(\det(M_2))$ would be a proper subgroup of $H$ that contains both $M_1$ and $M_2.$
The maximality of $M_1$ and $M_2$ then forces $M_1= M_2.$
So if $H$ is minimal, then every maximal subgroup of $H$ is normal in $H$. 
Further, if we let $p\in S$ , $k\in \ZZ_+$ and $\pi_{p^k}:H\to\GL_2(\ZZ/p^k\ZZ)$ be the standard component-wise reduction map on $H$, the $\pi_{p^k}(H)$ also has the property that all of its maximal subgroups are normal. 

Without giving too much background, we recall the definition of a nilpotent group.
\begin{definition}
	Let $G$ be a group. 
	We say that $G$ is {\bfseries nilpotent} if $G$ has a finite central series. 
\end{definition}
The following result demonstrates the relevance of this concept to $\pi_{p^{k}}(H)$.

\begin{theorem}{\rm \cite[Chapter 6, Theorem 3 \& Corollary 4]{DF}}\label{thm:finite_nilpotent}
Let $G$ be a finite group. 
The following are equivalent;
\begin{enumerate}
	\item The group $G$ is nilpotent.
	\item Every Sylow subgroup of G is normal.
	\item The group $G$ is the direct product of its Sylow subgroups.
	\item Every maximal subgroup of $G$ is normal in $G$. 
\end{enumerate}	
\end{theorem}

From this we immediately get the following proposition

\begin{prop}
If $H$ is a minimal subgroup of $\GL_2(\ZZ_S)$, then for every $p\in S$ and $k\in \ZZ_+$, we have that $\pi_{p^k}(H)$ is nilpotent. 
\end{prop}

Before we prove the main result of this section, we need one more definition.

\begin{defn}
	Let $G\subseteq \GL_2(\ZZ_S)$ be an open subgroup and given any $n\in \ZZ_+$ only divisible by primes in $S$, define $\rho_n : \GL_2(\ZZ_S)\to \GL_2(\ZZ/n\ZZ)$ to be the standard component-wise reduction map. 
	Suppose there is an $n\in \ZZ_+$ such that $G = \rho_n^{-1}(\rho_n(G))$. In this case we say that $G$ has {\bfseries finite level} and we define the {\bfseries level} of $G$ is the smallest $N\in \ZZ_+$ such that $G = \rho_N^{-1}(\rho_N(G))$. 
	If no such $n\in\ZZ_+$ exists, we say that $G$ has infinite level.
\end{defn}


\begin{remark}
The assumption that $H$ has finite level is not crucial. It is known by work of Nikolov and Segal \cite{NS} that every topologically finitely generated profinite group has the property that every finite index subgroup is open. 
\end{remark}

We are now ready to provide a proof of Theorem~\ref{thm:S=2}.

\begin{proof}[Proof of Theorem~\ref{thm:S=2}]
We prove this theorem by contradiction. 
Suppose that $S$ is a nonempty set of primes, $p$ is an odd prime such that $p\in S$, and $H$ is a minimal subgroup of $\GL_2(\ZZ_S)$ of finite level.

Since $H$ has finite level we know that there is a $k\in \ZZ_+$ such that $p^{k-1}$ divides the level of $H$, while $p^{k}$ does not. From this, it follows that $M \in \rho_{p^{k}}(H)$ if and only if $M \in \rho_{p^{k-1}}(H)$. In particular, $H_{p^{k}} = \rho_{p^{k}}(H)$ contains all matrices equivalent to the identity mod $p^{k-1}$.

As we noted before, we know that $H_{p^k}$ is a finite nilpotent group. 
From Theorem $\ref{thm:finite_nilpotent}$ we have that $H_{p^k}$ must be the direct product of its $p$-Sylow subgroups. 
That is, if we let $P = {\rm Syl}_p(H_{p^k})$ and $Q = \prod_{q\neq p} {\rm Syl}_q(H_{p^k})$, then $H_{p^k} \simeq P\times Q$. 
So given any $B \in H_{p^k}$ we can find $X\in Q$ and $Y\in P$ such that $B = XY$. We also know that $X$ and $Y$ commute with each other in $H_{p^k}\simeq P\times Q$ since we have $X\in Q$ and $Y\in P$ . 
Since $Y\in {\rm Syl}_p(H_{p^k})$, we know that the order of $Y$ is a power of $p$ and thus for some $j$ $\det(Y)^{p^{j}} = \det(Y^{p^j}) = \det(I) \equiv 1 \bmod p.$ Fermat's little theorem says $\det(Y)^{p^{j}} \equiv 1 \bmod p$ forces $\det(Y) \equiv 1 \bmod p$.

Next, notice that for any $Z\in M_2(\ZZ/p^k\ZZ)$ we have that $I+p^{k-1}Z\in H_{p^k}$ since $I+p^{k-1}Z \equiv I \bmod p^{k-1}$. 
Now, a simple computation shows that $I+p^{k-1}Z$, has order dividing $p$ and so it must be that $(I+p^{k-1}Z)\in P$. 
Since everything in $P$ commutes with the elements of $Q$ we have that
\[
X(I+p^{k-1}Z) = (I+p^{k-1}Z)X.
\]
An immediate consequence of this is that $$XZ \equiv ZX \bmod p.$$
Since $Z$ was an arbitrary element of $M_2(\ZZ/p^k\ZZ)$, we can see that $X \bmod p$ commutes with everything in $M_2(\ZZ/p\ZZ)$. 
An elementary computation shows that it must be that 
$$X \equiv \begin{pmatrix}
	\alpha&0\\0&\alpha
\end{pmatrix} \bmod p$$
for some $\alpha\in \ZZ/p\ZZ$ and so $\det(X) \equiv \alpha^2 \bmod p.$ 
Bringing it all together we get that 
$$\det(B) = \det(XY) = \det(X)\det(Y) \equiv \alpha^2 \bmod p.$$
Thus for every $B\in H$, $\det(B)$ is a quadratic residue in $\ZZ/p\ZZ$. 
Since $p$ is an odd prime, not every element of $\ZZ/p\ZZ$ is a quadratic residue. 
This contradicts the assumption that $\det(H) = \ZZ_S^\times$.
\end{proof}

\begin{remark}
One way to make sense of what is happening here is that we are requiring our minimal groups to have surjective determinant. Looking at $\ZZ_p^\times$, we see that there is structural difference depending on if $p=2$ or if $p$ is odd. When $p=2$, the group $\ZZ_2^\times$ is a 2-group, but when $p$ is odd, $\ZZ_p^\times$ is not a $p$-group. This is because $(\ZZ/p\ZZ)^\times$ has size $p-1$. This factor of $p-1$ is what prevents $\ZZ_p^\times$ from being a $p$-group when $p$ is odd, but there is no problem here when $p=2.$
\end{remark}

\subsection{A diversion into higher dimensional abelian varieties}

Elliptic curves are $1$-dimensional abelian varieties, and much of what is true for elliptic curves is also true for abelian varieties being careful to adjust the details where necessary. Given the aim of this paper, we will not be able to provide all of the background information here, but the interested reader is encouraged to see \cite{MilneAV}.

To start, we let $A$ be a principally polarized $g$-dimensional abelian variety defined over a number field $K$ disjoint from $\Q(\zeta_{n})$. Then because the Weil-pairing is a non-degenerate, alternating, Galois invariant bilinear form on $A[n]$, it follows that the mod $n$ Galois representations associated to $A$,
\[
\bar\rho_{A,n}:\Gal(\overline{K}/K)\to \Aut(A[n])\simeq \GL_{2g}(\ZZ/n\ZZ)
\]
actually has its image (up to conjugation) contained inside of ${\rm GSp}_{2g}(\ZZ/n\ZZ)$. To define this group, 
let $\Omega$ be the $2g \times 2g$ block matrix of the form
\[
 \Omega = \begin{bmatrix}
 0 & -I_{g}\\
 I_{g} & 0 \end{bmatrix},
\]
where $I_{g}$ is the $g \times g$ identity matrix. For a ring $R$, the group ${\rm GSp}_{2g}(R)$ can be defined as the set of $2g \times 2g$ matrices $M$, with entries in $R$ such that
\[
M^T\Omega M = \lambda \Omega
\]
for some $\lambda\in R^\times$. With this, we can define a map 
\[
{\rm Mult}\colon{\rm GSp}_{2g}(R) \to R^\times
\]
given by ${\rm Mult}(M) = \lambda$, where $M^T\Omega M = \lambda \Omega$. It turns out that in the case that $A$ is a principally polarized $g$-dimensional abelian variety defined over a number field $K$ disjoint from $\Q(\zeta_{n})$, properties of the Weil pairing imply that
\[
{\rm Mult} \circ \bar\rho_{A,n} : \Gal(\overline{K}/K) \to (\Z/n\Z)^{\times}\]
is surjective.

With this in hand, we can see that the argument in the proof of Theorem \ref{thm:S=2} generalizes to ${\rm GSp}_{2g}(\Z_{S})$. In particular, if we assume
that $A/K$ is an abelian variety of dimension $g$, $p$ is a prime, and the image of the mod $p^{k}$ Galois representation attached to $A$ is nilpotent and contains all matrices in ${\rm GSp}_{2g}(\Z/p^{k} \Z)$ that are congruent to the identity modulo $p^{k-1}$, then any $X$ in the image of $\bar\rho_{A,p^{k}}$ with order coprime to $p$ must commute with matrices $I + p^{k-1} Z \in {\rm GSp}_{2g}(\Z/p^{k} \Z)$.
Writing
\[
 X = \begin{bmatrix} X_{1} & X_{2} \\ X_{3} & X_{4} \end{bmatrix}
\]
and taking 
\[
 Z = \begin{bmatrix} A & 0 \\ 0 & -A^{T} \end{bmatrix} 
\]
for an arbitrary $A \in \GL_{g}(\Z/p\Z)$ shows that $X_{1} \equiv \lambda_{1} I_{g} \bmod p$
and $X_{4} \equiv \lambda_{2} I_{g} \bmod p$.

Taking 
\[
 Z = \begin{bmatrix} I_{g} & I_{g} \\ 0 & -I_{g} \end{bmatrix}
\]
shows that $X_{3} \equiv 0 \bmod p$ and taking
\[
 Z = \begin{bmatrix} I_{g} & 0 \\ I_{g} & -I_{g} \end{bmatrix}
\]
shows that $X_{2} \equiv 0 \bmod p$ and that $X_{1} \equiv X_{4} \bmod p$. It follows that
\[
 X \equiv \begin{bmatrix} \lambda I_{g} & 0 \\ 0 & \lambda I_{g} \end{bmatrix} \bmod p.
\]
which shows that ${\rm Mult}(X) \equiv \lambda^{2} \bmod p$.

Any element in the image of $\bar\rho_{A,p^{k}}$ must have the form $XY$
for some $Y$ with order a power of $p$. This implies that ${\rm Mult}(Y) \equiv 1 \bmod p$, and so
${\rm Mult}(XY) = {\rm Mult}(X) {\rm Mult}(Y) \equiv \lambda^{2} \bmod p$. If $p > 2$, this contradicts
that ${\rm Mult} \circ \bar\rho_{A,p}$ is surjective, and this is a contradiction.
{}
\subsection{A diversion into CM elliptic curves} 

In the definition of minimal, we assumed that $H \subseteq \GL_{2}(\ZZ_{S})$ was a finite index subgroup and 
$\det : H \to \ZZ_{S}^{\times}$ is surjective. This will be true if $H$ is the $S$-adic image of Galois for a non-CM elliptic curve $E$ defined over some number field $K$ with the property that $K \cap \Q(\zeta_{n}) = K$
for any positive integer $n$ all of whose prime factors are in $S$ (as proven in \cite{Serre68}). What about the CM case?

First, at no point in the argument in Subsection~\ref{sub:1Ex} do we need to exclude CM elliptic curves. Thus the argument applies also to the CM case and shows that if $E/\Q$ is a CM elliptic curve with bad reduction only at $2$, then $\Im \rho_{E,2^\infty} \subseteq \GL_{2}(\Z_{2})$ is a subgroup for which $\det(\Im \rho_{E,2^\infty}) = \Z_{2}^{\times}$ and for which every proper closed subgroup comes from the determinant. 

However, the argument in the proof of Theorem~\ref{thm:S=2} used in a crucial way that the image of the $S$-adic Galois representation had finite level, and this will not be true for CM curves. This raises the following open question.

\begin{question}
Is there a finite set of primes $S$ containing at least one odd prime, a number field $K$ with the property that $K \cap \Q(\zeta_{n}) = \Q$ for all positive integers $n$ all of whose prime factors are in $S$, and a CM elliptic curve $E/K$ for which every maximal closed subgroup of the $S$-adic image of Galois comes from $\Z_{S}^{\times}$?
\end{question}

\section{Minimal groups are plentiful}\label{sec:plentiful}

The goal of this section is to show that the minimal groups are plentiful inside of $\GL_2(\ZZ_2)$. We start with a lemma about how minimal groups are generated.

\begin{lemma}
If $G\subseteq \GL_2(\ZZ_2)$ is a minimal group, then $G$ must be (topologically) generated by 2 elements of $\GL_2(\ZZ_2)$. 
\end{lemma}

\begin{proof}
Suppose $G\subseteq \GL_2(\ZZ_2)$ is a minimal group. Then $G$ must have precisely three maximal closed subgroups, each of which is the preimage under $\det : G \to \ZZ_{2}^{\times}$ of one of the three maximal closed subgroups of $\ZZ_{2}^{\times}$.
\end{proof}

Our goal is to show that a ``randomly'' chosen two generator subgroup of $\GL_{2}(\ZZ_{2})$ with surjective determinant is minimal. To quantify this, recall that $\GL_{2}(\ZZ_{2})$ has a Haar measure which (because $\GL_{2}(\ZZ_{2})$ is compact) is both left and right invariant. The main question we must answer is the following. 

\begin{question}
If we randomly pick $A,B\in \GL_2(\ZZ_2)$, what is the probability that the topological closure of $\langle A,B \rangle$ is a minimal group?
\end{question}

For a finite set $S \subseteq \GL_{2}(\Z_{2})$, we write $\langle S \rangle$ for the closure of the subgroup generated by $S$.

\begin{thm}
\label{mainthm}
The set of pairs $(A,B) \in \GL_{2}(\Z_{2})^{2}$ for which
$\langle A, B \rangle$ has infinite index in $\GL_{2}(\Z_{2})$ has measure zero.
\end{thm}

\begin{lemma}
Suppose that 
\[
f(x_{1},x_{2},\ldots,x_{n}) = \sum_{\mathbf{s} \in \Z_{\geq 0}^{n}} a_{\mathbf{s}} x^{\mathbf{s}} 
\]
is a multivariable power series that converges on an open subset $D \subseteq \Z_{2}^{n}$
with the property that not all $a_{\mathbf{s}}$ are equal to zero.
Then the measure of
\[
\{ (x_{1},x_{2},\ldots,x_{n}) \in D : f(x_{1},x_{2},\ldots,x_{n}) = 0\}
\]
is zero with respect to natural $p$-adic Haar measure.
\end{lemma}

\begin{proof}\footnote{The proof of this lemma follows the argument in the MSE post \href{https://math.stackexchange.com/questions/3216833/holomorphic-function-on-mathbbcn-vanishing-on-a-positive-lebesgue-measure}{here},
which proves the result in $\C^{n}$.}
We prove this by induction on $n$. Let $\mu_{n}$ denote the usual measure on $\Z_{2}^{n}$.

For $n = 1$, the power series $f(x_{1})$ is a $p$-adic analytic function
on $D$, and it is well-known (see for example \cite[Section 6.2]{Robert})
that the zeros of an analytic function are isolated. From this it follows that $f(x_{1})$ has finitely many zeros in $\Z_{2}$ and a finite set has measure zero.\footnote{Strassmann's theorem gives a more explicit upper bound on the number of zeros of a $p$-adic power series.}

Now assume that the result is true for power series in $n-1$ variables and write
\[
f(x_{1},x_{2},\ldots,x_{n}) = \sum_{j=0}^{\infty} a_{j}(x_{1},\ldots,x_{n-1}) x_{n}^{j}.
\]
Let $E$ be the zero set of $f$ and let $\chi_{E}$ be the characteristic function of $E$. By the Fubini theorem, we have
\[
\mu_{n}(E) = \int_{\Z_{2}^{n-1}} \int_{\Z_{2}} \chi_{E}(\vec{x},x_{n}) \, dx_{n} \, d\vec{x}.
\]
If we assume that $\mu_{n}(E) > 0$, the integrand must be positive on a set of positive measure.
In particular, there is some $F \subseteq \Z_{2}^{n-1}$ so that $\mu_{n-1}(F) > 0$ and
for all $\vec{x} \in F$,
\[
\mu_{1}(\{ x_{n} : (\vec{x},x_{n}) \in E \}) > 0.
\]
So if $\vec{x} \in F$, there is a positive measure set of $x_{n}$ for which
\[
\sum_{j=0}^{\infty} a_{j}(\vec{x}) x_{n}^{j}
\]
vanishes. By the one-variable case, it follows that the one-variable power series
$\sum a_{j}(\vec{x}) x_{n}^{j}$ vanishes, and therefore
$a_{j}(\vec{x}) = 0$. This shows that each $a_{j}$ vanishes on a set of positive measure
and the $n-1$ variable case shows that $a_{j} = 0$ for all $j$. This shows that all the coefficients of $f$
are equal to zero and so $f = 0$.
\end{proof}

Next, we need to apply ideas from $p$-adic Lie theory. For an introduction to this material, see \cite{BookProP}.
Let $\Gamma_{2} = \left\{ M \in \GL_{2}(\Z_{2}) : M \equiv I \pmod{4} \right\}$. Suppose that $A, B \in \Gamma_{2}$ and $G = \langle A, B \rangle$. Because $\Gamma_{2}$ is a uniform $2$-group, the group $\Gamma_{2}$ can be given the structure of a $\Z_{2}$-Lie algebra via the operations
\[
x + y = \lim_{n \to \infty} (x^{p^{n}} y^{p^{n}})^{p^{-n}}
\]
and the Lie bracket
\[
(x,y) = \lim_{n \to \infty} [x^{p^{n}},y^{p^{n}}]^{p^{-2n}}
\]
(where $[a,b] = aba^{-1} b^{-1}$).

Let $\mathfrak{gl}_{2}(R)$ denote the usual Lie algebra of $2 \times 2$ matrices with entries in $R$, where
 the Lie bracket is given via $[X,Y] = XY - YX$. Theorem 7.13 of \cite{BookProP} shows that the logarithm map sending $G \to \log(G) \subseteq \mathfrak{gl}_{2}(4 \Z_{2})$ is a Lie algebra isomorphism.

If $\log(G)$ has rank $4$ as a $\Z_{2}$-module, then there is a positive integer $k \geq 2$ so that
$\log(G)$ contains any matrix $\equiv 0 \pmod{2^{k}}$. If $X \equiv 0 \pmod{2^{k}}$, then
$\log(I+X) \equiv 0 \pmod{2^{k}}$. 
Hence $\log(I+X) \in \log(G)$ and thus
$I+X = \exp(\log(I+X)) \in G$. This shows that $G$ has finite index.

\begin{proof}[Proof of Theorem~\ref{mainthm}]
Fix two matrices $\overline{A}$ and $\overline{B}$ in $\GL_{2}(\Z/4\Z)$ and let
\begin{align*}
A &= \overline{A} + \begin{bmatrix} 4a_{1} & 4a_{2} \\ 4a_{3} & 4a_{4} \end{bmatrix}\\
B &= \overline{B} + \begin{bmatrix} 4b_{1} & 4b_{2} \\ 4b_{3} & 4b_{4} \end{bmatrix}.
\end{align*}
Because the exponent of $\GL_{2}(\Z/4\Z)$ is $12$, we have that $\overline{A}^{12} \equiv \overline{B}^{12} \equiv I \pmod{4}$.

Let $M$ be the $4 \times 4$ matrix whose columns consist of the entries of
\[
\log(A^{12}), \log(B^{12}), [\log(A^{12}),\log(B^{12})], [[\log(A^{12}),\log(B^{12})],\log(A^{12})],
\]
and let $d = \det(M)$. By the discussion above, if $d \ne 0$, then the four matrices above,
which are all in $\log(\langle A^{12}, B^{12} \rangle)$, are linearly independent over $\Z_{2}$ which
implies that $\log(\langle A^{12}, B^{12} \rangle)$ is a free $\Z_{2}$-module of rank $4$, and this implies that $\langle A^{12}, B^{12} \rangle$ has finite index in $\GL_{2}(\Z_{2})$ (which implies that $\langle A, B \rangle$ also has finite index).

This $d$ is a power series in the $8$ variables $a_{1}, a_{2}, a_{3}, a_{4}, b_{1}, b_{2}, b_{3}, b_{4}$ which only depends on $\overline{A}$ and $\overline{B}$. Moreover, because
$A^{12} \equiv I \pmod{4}$ and $B^{12} \equiv I \pmod{4}$, $A^{12} - I$ is a polynomial
in $4a_{1}$, $4a_{2}$, $4a_{3}$ and $4a_{4}$ and the formula for the logarithm implies that
$\log(A^{12})$ and $\log(B^{12})$ are power series in $\{ 4a_{1}, 4a_{2}, 4a_{3}, 4a_{4} \}$
and $\{ 4b_{1}, 4b_{2}, 4b_{3}, 4b_{4} \}$ respectively. This ensures that $d$ converges on all of
$\Z_{2}^{8}$.

The result will follow if we can show that for each pair $(\overline{A},\overline{B}) \in \GL_{2}(\Z/4\Z)^{2}$, the power series $d$ is nonzero. It suffices to show that this power series $d$ has at least one nonzero specialization. We check this computationally by randomly choosing the $a_{1}, \ldots, b_{4} \in \{ 1, 2, 3 \}$ and computing $d$ and checking if it is nonzero. In all $96^{2}$ cases, we find a case where $d \not\equiv 0 \pmod{2^{50}}$. This completes the proof.
\end{proof}

\begin{proof}[Proof of Theorem~\ref{thm:lots}]
Suppose that $H \subseteq \GL_{2}(\ZZ_{2})$ is a given subgroup with $\det(H) = \ZZ_{2}^{\times}$. If $A \in \GL_{2}(\ZZ_{2})$ is a randomly chosen matrix with $\det(A) \equiv 3 \pmod{8}$ and $B \in \GL_{2}(\ZZ_{2})$ is a randomly chosen matrix with $\det(B) \equiv 5 \pmod{8}$, let $M = \langle A, B \rangle$. This is a subgroup of $\GL_{2}(\ZZ_{2})$ with $\det(M) = \ZZ_{2}^{\times}$ and precisely three maximal closed subgroups. With probability $1$, $M$ has finite-index in $\ZZ_{2}^{\times}$, and with probability $\frac{1}{|\GL_{2}(\ZZ_{2}) : H|^{2}}$, we have $M \leq H$. So there is a positive probability that $M$ is a minimal subgroup of $H$, and therefore minimal subgroups of $H$ exist. 
\end{proof}

\section{Genus 0 Examples}\label{sec:Gen0}

\subsection{The search for genus 0 groups}

We now turn our attention to finding all of the genus 0 minimal subgroups of $\GL_{2}(\ZZ_{2})$ up to conjugacy. We focus on genus zero curves because they supply infinitely many examples of elliptic curves whose 2-adic representations are minimal. Our first challenge is to find a finite box that contains all of the genus 0 minimal groups. According to \cite{CPData,CP}, the largest index of a subgroup of $\PSL_2(\ZZ_2)$ with genus $0$ is $48$. If $G \subseteq \GL_{2}(\ZZ_{2})$ is a minimal group with $|\GL_{2}(\ZZ_{2}) : G| = d$, then surjectivity of the determinant gives that $|\SL_{2}(\ZZ_{2}) : G \cap \SL_{2}(\ZZ_{2})| = d$ and this implies that the image of $G \cap \SL_{2}(\ZZ_{2})$ in $\PSL_{2}(\ZZ_{2})$ has index $d$ or $d/2$. It follows that $d \leq 96$.

Next, if $P$ is a Sylow $2$-subgroup of $\GL_{2}(\ZZ_{2})$, then $|\GL_{2}(\ZZ_{2}) : P| = 3$. Each maximal closed subgroup of $P$ has level at most $8$ and index $6$. If $H \subseteq \GL_{2}(\ZZ_{2})$ has level
$2^{k}$ with $k \geq 2$, then every maximal subgroup of $H$ has level dividing $2^{k+1}$ by Lemma 3.3 of \cite{RZB}. Hence $|GL_{2}(\ZZ_{2}) : G| \leq 96$ implies that the level of $G$ is at most $128$.

Searching for subgroups of $\GL_2(\ZZ/128\ZZ)$ yields 7652 minimal groups, of which 28 have genus 0. Eight of these were already listed in Section~\ref{sec:intro}. For the other 20 minimal groups $G$, the modular curve $X_{G}$ is a conic without rational points. None of these $28$ genus $0$ minimal groups contain $-I$. In fact, none of the minimal groups of level $\leq 128$ contain $-I$. We conjecture that it is impossible for a minimal subgroup of $\GL_{2}(\ZZ_{2})$ to contain $-I$.

\subsection{Models for the corresponding modular curves}

The models for the 28 genus 0 modular curves corresponding to elliptic curves with minimal 2-adic image were computed using a combination of techniques. The 8 modular curves that were found using pure thought were computed in \cite{RZB} and we simplified the models.  The remaining 20 models were computed by using the techniques in \cite{rakvi} to compute covering models for modular curves associated to supergroups of these groups that contain $-I$. We then use this information together with information from \cite{RZB} and compute models for these curves as fiber products of known modular curves.

Returning to the 8 minimal groups $G$ that we found using pure thought in Section \ref{sec:intro}, if $E_{1}$ is the universal elliptic curve over $X_{G}$, then 3 other universal elliptic curves can be obtained by taking of twists of $E_{1}$ by $-1$, $2$, and $-2$. 
We remark here that we are guaranteed that if $E$ is an elliptic curve whose 2-adic image is minimal, then the twists of $E$ by $-1$, $2$ and $-2$ will also have minimal image because $\QQ(i)$, $\QQ(\sqrt{2})$ and $\QQ(\sqrt{-2})$ are the quadratic subfields of $\QQ(\zeta_{2^\infty})\subseteq \QQ(E[2^\infty]).$ 
The other 4 examples are $2$-isogenous to the curves obtained this way. It turns out that this relationship persists for most of the other genus $0$ minimal groups without points defined over $\QQ$. That relationship is made explicit below. 
The vertical maps are twists by 2, the horizontal maps are twists by $-1$, and the maps from front to back are $2$-isogenies. 

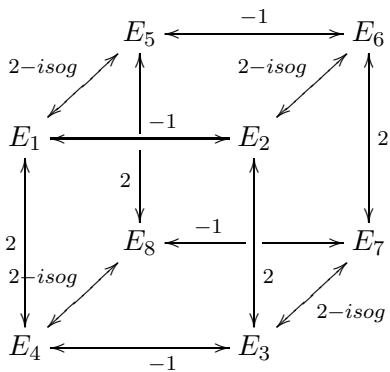
\[
\xymatrix{
& E_5 \ar@{<->}[rr]^{-1} \ar@{<-}[d] && E_6 \ar@{<->}[dd]^2 \\
E_1 \ar@{<->}[ru]^{2-isog} \ar@{<->}[rr]^(.6){-1} \ar@{<->}_2[dd] && E_2 \ar@{<->}[ru]^{2-isog} \ar@{<->}[dd]^(.66)2 & \\
& E_8 \ar@{<-}[u]^(.6)2\ar@{<-}[r]^(.6){-1} && E_7\ar@{<-}[l] \\ E_4 \ar@{<->}[ru]^{2-isog} \ar@{<->}[rr]_(.6){-1} && E_3 \ar@{<->}[ru]_{2-isog} &
}
\]

\medskip

\[
\begin{tabular}{|l||c|c|c|}\hline
$E_1$ & \texttt{16.48.0.25} & \texttt{32.96.0.2} & \texttt{32.96.0.1} \\\hline
$E_2$ & \texttt{16.48.0.26} & \texttt{32.96.0.4} & \texttt{32.96.0.3} \\\hline
$E_3$ & \texttt{16.48.0.82} & \texttt{32.96.0.26} & \texttt{32.96.0.25} \\\hline
$E_4$ & \texttt{16.48.0.83} & \texttt{32.96.0.28} & \texttt{32.96.0.27} \\\hline
$E_5$ & \texttt{16.48.0.238} & \texttt{32.96.0.105} & \texttt{32.96.0.108} \\\hline
$E_6$ & \texttt{16.48.0.239} & \texttt{32.96.0.107} & \texttt{32.96.0.106} \\\hline
$E_7$ & \texttt{16.48.0.234} & \texttt{32.96.0.101} & \texttt{32.96.0.104} \\\hline
$E_8$ & \texttt{16.48.0.235} & \texttt{32.96.0.103} & \texttt{32.96.0.102} \\\hline
\end{tabular}
\]

\medskip

The remaining 4 cases have RSZB label \texttt{8.24.0.44}, \texttt{8.24.0.86}, \texttt{8.24.0.123}, and \texttt{8.24.0.125}. If $X_{G}$ is one of these modular curves, there is a universal elliptic curve $E_{1}$ parametrizing elliptic curves with $2$-adic image contained in $G$. These elliptic curves are defined over
$F = \Q(X_{G})$, the function field of a pointless conic, and there is an automorphism $\varphi:F\to F$ with the property that $\varphi(E)\simeq E^{-1}$, where $E^{-1}$ is the quadratic twist of $E$ by $-1$.  The existence of this automorphism collapses the cube above down to a square. 
The data can be summarized as follows:

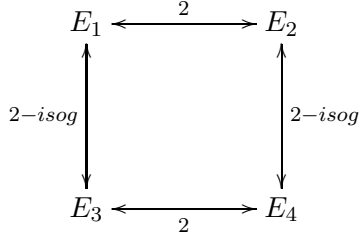
\[
\xymatrix{
E_1 \ar@{<->}[rr]^2 \ar@{<->}[dd]_{2-isog} & & E_2 \ar@{<->}[dd]^{2-isog}\\
& & \\
E_3 \ar@{<->}[rr]_2& & E_4 \\
}
\]

\[
\begin{tabular}{|l||c|}\hline
$E_1$ & \texttt{8.24.0.44} \\\hline
$E_2$ & \texttt{8.24.0.86} \\\hline
$E_3$ & \texttt{8.24.0.125} \\\hline
$E_4$ & \texttt{8.24.0.123} \\\hline
\end{tabular}
\]

For each of the families of elliptic curves we provide simplified models for $E_{1}$. 
The interested user can then take twists and isogenies to recover the remaining models. See Remark \ref{rmk:compute} for more details. 
For each curve we give a model for, we will define the base field as well as provide an $A$ and $B$ from the base field such that a generic elliptic curve with the corresponding image has the form
\[
y^2 =x^3+Ax^2+Bx.
\]
For three of the four families of curves, the base field they are defined over is the field of fractions of $\QQ[a,b]/(a^2+b^2+1)$ while the remaining curve is defined over $\QQ(t)$.

\begin{remark}
We pause of a moment and give a quick summary of what is happening here. What we know is that every elliptic curve (up to isomorphism) over a number field $K$ in which $-1$ is the sum of two squares and image in \texttt{16.48.0.25}, \texttt{39.96.0.2}, or \texttt{8.24.0.44} can be obtained by choosing $a$ and $b$ in $K$ with $a^2+b^2 = -1$ and plugging those $a$ and $b$ into our formulas. For \texttt{32.96.0.1}, every curve (up to isomorphism) defined over $\QQ$ with this image can be obtained by choosing a $t\in\QQ$ and plugging it into our formulas.
\end{remark}

\[
\begin{tabular}{|c|r|l|}\hline
\renewcommand{\arraystretch}{1.2}
\multirow{2}{*}{\texttt{16.48.0.25}} &A & \footnotesize{$2^2(b^2 - 2b - 1) (b^2 + 2b - 1) $}\\
&B &\footnotesize{$2^3(b^2 + 1)^2(b^2 + 2b - 1)^2 $}\\\hline

\multirow{2}{*}{\texttt{32.96.0.2}} &A &\footnotesize{$2^6(b^2 - 3)(b^4 - 22b^2 - 7)(b^8 + 116b^6 + 1462b^4 + 4372b^2 + 3281)a-$}\\ & & \footnotesize{$ 
2^2(b^8 - 108b^6 + 790b^4 + 116b^2 - 527)(b^8 + 116b^6 + 1462b^4 + 4372b^2 + 3281) $} \\
&B &\footnotesize{$ -2^7(b^2-3)^4(3b^{10} - 101b^8 - 850b^6 + 5126b^4 + 5983b^2 - 913)(b^{12} - 210b^{10} + 455b^8 +$}\\ & & \footnotesize{$\phantom{-2^7(b^2-3)^4(} 27236b^6+ 2879b^4 - 62834b^2 - 35047)a$}\\ & & \footnotesize{$ + 
2^3(b^2-3)^5(b^{22} - 993b^{20} + 80239b^{18} - 183591b^{16} - 25060758b^{14} - 46958090b^{12} +
$}\\ & &\footnotesize{$\phantom{+ 2^3(b^2-3)^5(} 1283004574b^{10} + 3556278098b^8 + 2155079365b^6 - $}\\
& &\phantom{$+b^2-3)5$} \footnotesize{$1522506117b^4 - 
1813927741b^2 - 391647931)$}\\\hline

\multirow{2}{*}{\texttt{8.24.0.44}} &A &\footnotesize{$2(b^2 + 1)(64b^4 - 16b^3 + 144b^2 - 16b + 79)(8b^2 + 7) $} \\
&B &\footnotesize{$ 2^4(b^2+1)^3(32768b^{10} - 20480b^9 + 245760b^8 - 92160b^7 + 660480b^6 - 153216b^5 + 
$}\\ & & \footnotesize{\phantom{$2^4(b^2+1)^3$}$833280b^4 - 111840b^3+ 504360b^2 - 30305b + 118568)$}\\\hline

\multirow{2}{*}{\texttt{32.96.0.1}} &A &\footnotesize{$-2^2(t^{16} - 120t^{14} + 1820t^{12} - 8008t^{10} + 12870t^8 - 8008t^6 + 1820t^4 - 120t^2 + 1) $ }\\
&B &\footnotesize{$2^3(t^2+1)^8(t^8 - 8t^7 - 28t^6 + 56t^5 + 70t^4 - 56t^3 - 28t^2 + 8t + 1)^2 $}\\\hline

\end{tabular}
\]

\begin{remark}\label{rmk:compute}
Each of the elliptic curves above has a unique rational 2-isogeny. The kernel of this isogeny is exactly $\{\mathcal{O},(0,0)\}.$ A classical result tells us that if $E$ is an elliptic curve of the form 
\[
y^2 = x^3+Ax^2+Bx,
\]
and $\varphi\colon E\to E'$ is the 2-isogeny with kernel $\{\mathcal{O},(0,0)\}$, then $E'$ is given by
\[
y^2 = x^3-2Ax^2+(A^2-4B)x.
\]
Using this and \cite[Proposition 5.4]{Silv1}, we can easily compute the models for the other elliptic curves in each family. 
\end{remark}

\section{Elliptic curves over imaginary quadratic fields with minimal image}\label{sec:ExQF}

In this final section, we give a family of elliptic curves defined over imaginary quadratic fields with minimal $2$-adic image and bad reduction only at $2$. 

\begin{prop}
Suppose that $n$ is a positive integer and let $a = \sqrt{-(2^{n}+1)}$. Let
\[
E : y^{2} = x^{3} + 2ax^{2} + (a^{2} + 1)x.
\]
Then $E$ has bad reduction only at prime ideals above $2$ in $\Q(a)$. 
If $n$ is odd and $n \ne 3$, then $\rho_{E,2^\infty}(G_{\Q(a)})$ is minimal and has RSZB label \texttt{8.24.0.86}. If $n = 2$, then $\rho_{E,2^\infty}(G_{\Q(a)})$ is minimal and has RSZB label \texttt{16.384.9.895}. If $n = 10$, then $\rho_{E,2^\infty}(G_{\Q(a)})$ is minimal and has RSZB label \texttt{16.384.9.894}.
\end{prop}


\begin{proof}
The discriminant $\Delta(E) = -64(a^{2}+1)^{2} = -64 (-2^{n})^{2} = -2^{2n+6}$ is a power of $2$, and therefore the only primes at which $E$ could have bad reduction are primes above $2$. This proves the first claim.

A computation with models of $2$-adic modular curves in \cite{RZB} shows that $y^{2} = x^{3} + 2tx^{2} + (t^{2}+1)x$ is the universal elliptic curve with $2$-adic image \texttt{4.12.0.12}. 
It follows that for any $n \geq 1$, the $2$-adic image is contained in \texttt{4.12.0.12}. Note that this level 4 subgroup does not contain $-I$ and hence $-I \not\in \rho_{E,2^\infty}(\Q(a))$ for any $n$.

Suppose that $n$ is odd. Note that $2^{n} + 1$ cannot be twice a square. Moreover, $2^{n} + 1$ is a square if and only if $n = 3$. 
This ensures that $\Q(a) \not\in \{ \Q(i), \Q(\sqrt{-2}), \Q(\sqrt{2}) \}$.

There are four maximal subgroups of \texttt{4.12.0.12} with surjective determinant, and they are
\texttt{4.24.0.9}, \texttt{4.24.0.10}, \texttt{8.24.0.85}, and \texttt{8.24.0.86}. 
Using a model for
\texttt{8.12.0.40} (generated by \texttt{8.24.0.86} and $-I$) we see that the $2$-adic image $y^{2} = x^{3} + 2tx^{2} + (t^{2} + 1)x$ is contained in \texttt{8.24.0.86} if and only if $t^{2} + 1 = -2u^{2}$. This certainly occurs if $t = a = \sqrt{-(2^{n} + 1)}$ with $n$ odd. 
Moreover, \texttt{8.24.0.86} is a minimal group, and so if $\rho_{E,2^\infty}(G_{\Q(a)})$ is contained in it and has surjective determinant, it must equal it.

Suppose now that $n = 2$. Let $E' : y^{2} = x^{3} + 2a^{2} x^{2} + a^{2} (a^{2} + 1)x$. 
Then $E' : y^{2} = x^{3} - 10x^{2} + 20x$ is a quadratic twist of $E$ (by an element of $\Q(a)$). 
We know from the LMFDB that $\rho_{E',2^\infty}(G_{\Q})$ is \texttt{16.96.3.338}. This group and all of its index $2$ subgroups contain $-I$. This implies that $\rho_{E',2^\infty}(G_{\Q(a)})$ contains $-I$,
but we know that $\rho_{E,2^\infty}(G_{\Q(a)})$ does not. This implies that $\rho_{E,2^\infty}(G_{\Q(a)})$ is an index
$2$ subgroup of $\rho_{E',2^\infty}(G_{\Q(a)})$, which in turn implies that $\rho_{E',2^\infty}(G_{\Q(a)})$ must be a proper subgroup of $\rho_{E',2^\infty}(G_{\Q})$. There are four index two subgroups with surjective determinant up to conjugacy, and only one of these is contained in \texttt{4.6.0.5} (the subgroup generated by \texttt{4.12.0.12} and $-I$). This subgroup is \texttt{16.192.9.211} and so this must be $\rho_{E',2^{\infty}}(G_{\Q(a)})$. 
The group \texttt{16.192.9.211} has several index $2$ subgroups that do not contain $-I$ and only one of these is contained in \texttt{4.12.0.12}. This subgroup has label \texttt{16.384.9.895}, and is minimal.

We apply the same process with $n = 10$. 
Let $E' : y^{2} = x^{3} + 2a^{2} x^{2} + a^{2} (a^{2} + 1)x$. 
Then $E' : y^{2} = x^{3} - 2050x^{2} + 1049600x$ is a quadratic twist of $E$ (by an element of $\QQ(a)$). 
Since $j(E') = 257^{3}/2^8$, one of the exceptional $j$-invariants from the $2$-adic classification,
it follows that $\rho_{E',2^\infty}(G_{\QQ})$ is \texttt{16.96.3.335}. Once again, this subgroup has no
index $2$ subgroups that do not contain $-I$, and therefore $\rho_{E',2}(G_{\Q(a)})$ must. However,
$\rho_{E,2^\infty}(G_{\Q(a)})$ does not, and this implies that $\rho_{E,2^\infty}(G_{\Q(a)})$ is an index two subgroup of
$\rho_{E',2^\infty}(G_{\Q(a)})$. There are four index two subgroups of \texttt{16.96.3.335} with surjective determinant up to conjugacy, and only one of these is contained in \texttt{4.6.0.5}. This subgroup is \texttt{16.192.9.208} and so this must be $\rho_{E',2^{\infty}}(G_{\Q(a)})$. 
The group \texttt{16.192.9.208} has two index $2$ subgroups without $-I$ up to conjugacy and the only one
which is contained in \texttt{4.12.0.12} is the minimal \texttt{16.384.9.894}, which must equal
$\rho_{E,2^\infty}(G_{\Q(a)})$.
\end{proof}

\bibliographystyle{plain} 
\bibliography{bib}

\end{document}